\documentclass{amsart}
\usepackage{amssymb}

\input xy
\xyoption{all}

\newcommand{\IR}{\mathbb R}
\newcommand{\IN}{\mathbb N}

\newcommand{\pr}{\mathrm{pr}}
\newcommand{\upa}{{\uparrow}}

\newcommand{\II}{\mathbb I}
\newcommand{\Int}{\operatorname{Int}}
\newcommand{\Ra}{\Rightarrow}
\newcommand{\Hom}{\mathrm{Hom}}
\newcommand{\w}{\omega}
\newcommand{\Aff}{\mathrm{Aff}}
\newcommand{\U}{\mathcal U}

\newtheorem{theorem}{Theorem}[section]
\newtheorem{proposition}[theorem]{Proposition}

\newtheorem{corollary}[theorem]{Corollary}

\newtheorem{problem}[theorem]{Problem}
\theoremstyle{definition}
\newtheorem{example}[theorem]{Example}

\title[Semigroups embeddable into products of cones over topological groups]{On topological Clifford  semigroups embeddable into products of cones over topological groups}
\author{Taras Banakh and Iryna Pastukhova}
\address{T.~Banakh: Ivan Franko University of Lviv (Ukraine) and Jan Kochanowski University in Kielce (Poland)}
\email{t.o.banakh@gmail.com}
\address{I.~Pastukhova: Ivan Franko University of Lviv (Ukraine)}
\email{irynkapastukhova@gmail.com}
\thanks{The first author has been partially financed by NCN grant
DEC-2012/07/D/ST1/02087.}

\begin{document}

\begin{abstract} In this paper we detect topological Clifford semigroups which are embeddable into
Tychonoff products of topological semilattices and cones over topological groups.
Also we detect topological Clifford semigroups which embed into compact topological Clifford semigroups.
\end{abstract}
\maketitle

\section{Introduction}
This paper was motivated by the problem of recognizing topological Clifford semigroups which can be embedded into compact topological Clifford semigroups. We shall resolve this problem for topological Clifford $U$-semigroups (i.e., topological Clifford semigroups whose idempotent band $E$ is a $U$-semilattice). Topological $U$-semilattices and their relation to other classes of topological semilattices will be discussed in Section~\ref{topsem}. In Theorem~\ref{precomp} we shall prove that a topological Clifford $U$-semigroup $S$ embeds into a compact topological Clifford semigroup if and only if $S$ is ditopological, the maximal semilattice  $E$ of $S$ embeds into a compact topological semilattice and each maximal subgroup $H_e$, $e\in E$, of $S$ embeds into a compact topological group. Ditopological Clifford semigroups were introduced and studied in \cite{BP}. The class of such semigroups contains all compact topological Clifford  semigroups, all topological semilattices, all topological groups and is closed under many operations over topological Clifford semigroups. We shall discuss ditopological Clifford semigroups in Subsection~\ref{ditop}. In Section~\ref{2emb} we shall prove that each ditopological Clifford $U$-semigroup $S$ embeds into a Tychonoff product of topological semilattices and cones over maximal subgroups $H_e$, $e\in E$, of $S$.

\section{Preliminaries} In this section we recall some definitions and prove some auxiliary results.

\subsection{Topological spaces} All topological spaces considered in this paper are {\em Hausdorff}.
A topological space $X$ is called
\begin{itemize}
\item {\em zero-dimensional} if closed-and-open subsets form a base of the topology of $X$;
\item {\em punctiform} if each non-empty compact connected subset of $X$ is a singleton.
\end{itemize}
It is clear that each zero-dimensional topological space is punctiform. By
\cite[1.4.5]{Eng}, a locally compact space is zero-dimensional if
and only if it is punctiform.

The {\em weight} $w(X)$ of a topological space $X$ is the smallest cardinality $|\mathcal B|$ of a base $\mathcal B$ of the topology of $X$.

\subsection{Semigroups} A {\em semigroup} is a set $S$ endowed with an associative binary operation $*:S\times S\to S$. In the sequel we shall often omit the symbol $*$ of the operation and write $xy$ instead of $x*y$.

A semigroup $S$ is called
\begin{itemize}
\item {\em a semilattice} if $xy=yx$ and $xx=x$ for all $x,y\in S$;
\item {\em regular} if for each $x\in X$ there is $y\in X$ such that $xyx=x$ and $yxy=y$;
\item {\em an inverse semigroup} if for each $x\in X$ there exists a unique element $x^{-1}\in S$ such that $xx^{-1}x=x$ and $x^{-1}xx^{-1}=x^{-1}$;
\item {\em a Clifford semigroup} if $S$ is inverse and $xx^{-1}=x^{-1}x$ for all $x\in X$.
\end{itemize}
It is well-known \cite[5.1.1]{Howie} that a semigroup $S$ is inverse if and only if $S$ is regular and the set $E=\{e\in S:ee=e\}$
is a commutative subsemigroup of $S$. In this case $E$ will be called the {\em maximal semilattice} of $S$. If $S$ is a Clifford semigroup, then the map
$$\pi:S\to E,\;\;\pi:x\mapsto xx^{-1}=x^{-1}x,$$is a semigroup homomorphism, and for every idempotent $e\in E$ the
preimage $\pi^{-1}(e)$ coincides with the maximal subgroup $H_e=\{x\in S:xx^{-1}=e=x^{-1}x\}$ of $S$ containing the idempotent $e$.

A subset $I$ of a semigroup $S$ is called {\em an ideal} if $SIS\subset I$.

\subsection{Topological semigroups} By a {\em topological} ({\em inverse}) {\em semigroup} we understand an (inverse) semigroup $S$ endowed with a topology such that the semigroup operation $*:S\times S\to S$ is continuous (and the inversion $(\cdot)^{-1}:S\to S$ is continuous). If the inverse semigroup $S$ is Clifford, then $S$ will be called a {\em topological Clifford semigroup}.

\begin{proposition}\label{p2.1} A compact topological semigroup $S$ is Clifford if and only if $S$ contains a dense Clifford subsemigroup.
\end{proposition}

\begin{proof} Let $X\subset S$ be a dense Clifford subsemigroup of $S$. Let $E=\{e\in S:ee=e\}$ be the set of idempotents in $S$ and $E_X=X\cap E$. By the commutativity of the semigroup $E_X$ and the continuity of the semigroup operation $*:S\times S\to S$, the closure $\bar E_X$ of $E_X$ in $S$ is a semilattice.

Since the semigroup $X$ is Clifford, for every $x\in X$ there is a unique element $x^{-1}\in X$ such that $xx^{-1}x=x$, $x^{-1}xx^{-1}=x^{-1}$ and $xx^{-1}=x^{-1}x$. Consider the set $D_X=\{(x,x^{-1}):x\in X\}$ and its closure $\bar D_X$ in $S\times S$. It follows that $\bar D_X\subset\{(x,y)\in S\times S:xyx=x,\;\;yxy=y,\;\;xy=yx\}$. The continuity of the semigroup operation $*:S\times S\to S$ and the inclusion $*(D_X)\subset E_X$ imply $*(\bar D_X)\subset \bar E_X$.

Let $\pr_1:S\times S\to S$, $\pr_1:(x,y)\mapsto x$, be  the coordinate projection. The projection $\pr_1(\bar D_X)\supset X$, being a dense closed subset of $S$, coincides with $S$. Consequently, for each point $x\in S$ there is a point $x^\star\in S$ such that $(x,x^\star)\in\bar D_X$ and hence $xx^\star x=x$, $x^\star xx^\star=x^\star$ and $xx^\star=x^\star x$. This means that the semigroup $S$ is regular.

 We claim that $E=\bar E_X$.
Given any idempotent $e\in E$, find an element $e^\star\in S$ such that $(e,e^\star)\in \bar D_X$.
Taking into account that $ee^\star e=e$, $e^\star ee^\star=e^\star$ and $ee^\star=e^\star e$, we conclude that $e=ee=(ee^\star e)(ee^\star e)=(eee^\star)(eee^\star)=(ee^\star)(ee^\star)=ee^\star\in *(\bar D_X)\subset \bar E_X$.
Therefore, $E=\bar E_X$ is a commutative semigroup. Being a regular semigroup with commuting idempotents, the semigroup $S$ is inverse.
To show that $S$ is Clifford, it suffices to show that the set $E=\bar E_X$ is contained in the center $Z(S)=\{z\in S:\forall x\in S\;\;xz=zx\}$ of $S$.
Observe that for every $x\in X$ the set $Z_x=\{z\in S:zx=xz\}$ is closed in $S$ and contains the set $E_X$. Then $E=\bar E_X\subset Z_x$ for all $x\in X$.
 Now observe that for every $e\in E$ the set $Z_e=\{x\in X:xe=ex\}$ is closed and contains the dense subset $X$ of $S$. Consequently,
 $Z_e=S$ for all $e\in E$ and hence $E$ is contained in the center of $S$. By \cite[II.2.6]{Pet}, the inverse semigroup $S$ is Clifford.
\end{proof}

\begin{problem} Assume that a compact topological semigroup $S$ contains a dense inverse subsemigroup. Is $S$ inverse?
\end{problem}

We shall say that a topological semigroup $X$ {\em embeds} into a topological semigroup $Y$ if there exists a semigroup
homomorphism $h:X\to Y$, which is a topological embedding. In this case we shall write $X\hookrightarrow Y$.

A topological semigroup $S$ is defined to be {\em precompact} if $S$ embeds into a compact topological semigroup.

\begin{proposition}\label{compC} A topological Clifford semigroup $S$ is precompact if and only if $S$
embeds into a compact topological Clifford semigroup $K$ having
weight $w(K)=w(S)$.
\end{proposition}

\begin{proof} The ``if'' part is trivial. To prove the ``only if'' part, assume that $S$
is precompact and hence embeds into a compact topological semigroup $K$.
We lose no generality assuming that $S\subset K$. By Theorem 2.29
\cite{CHK1}, the compact topological semigroup $K$ embeds into a
Tychonoff product $\prod_{\alpha\in A}K_\alpha$ of metrizable
topological semigroups $K_\alpha$, $\alpha\in A$. By a standard
argument one can show that there is a subset $B\subset A$ of
cardinality $|B|\le w(S)$ such that the projection $\pr_B:S\to
\prod_{\alpha\in B}K_\alpha$ is a topological embedding. Then
$K_B= \prod_{\alpha\in B}K_\alpha$ is a compact topological
semigroup of weight $w(K_B)\le |B|\le w(S)$ containing a
topological copy of $S$. Without loss of generality we can assume
that $S\subset K_B$. By Proposition~\ref{p2.1}, the closure $\bar
S$ of $S$ in $K_B$ is a Clifford semigroup. By
Koch-Wallace-Kruming Theorem \cite{KW}, \cite{Krum} the inversion
$(\,)^{-1}:\bar S\to \bar S$ is continuous, which means that $\bar
S$ is a compact topological Clifford semigroup.
\end{proof}

In Section~\ref{precomp} we shall consider the problem of detecting precompact topological Clifford semigroups in more details.

Let $X,Y$ be two topological semigroups. By $\Hom(X,Y)$ we denote the set of all continuous semigroup homomorphisms from $X$ to $Y$. We shall say that $X$ is {\em $Y$-separated} (resp. {\em $Y$-embeddable}) if the {\em canonical homomorphism}
$$X\to Y^{\Hom(X,Y)},\;\;\;x\mapsto (h(x))_{h\in\Hom(X,Y)},$$is injective (resp. a topological embedding).

If a topological semigroup $Y$ is precompact, then so is each $Y$-embeddable semigroup $X$.

\subsection{Topological semilattices}\label{topsem}

A {\em topological semilattice} is a topological space $E$ endowed with a continuous semilattice operation. An important example of a topological semilattice is the unit
interval $[0,1]$ endowed with the usual topology and the operation of minimum
$(x,y)\mapsto \min\{x,y\}$. We denote
this topological semilattice by $\II$. The topological semilattice $\II$ contains a two-element subsemilattice $\mathbf 2=\{0,1\}$.

Each topological semilattice $E$ carries a partial order defined by $x\le y$ if $xy=x=yx$.
For a point $x\in E$ let ${\downarrow}x=\{y\in E:y\le x\}$ and ${\uparrow}x=\{y\in E:x\le y\}$ be the lower and upper cones of $x$, respectively.
For a subset $A$ of a semilattice $E$ we put ${\downarrow}A=\bigcup_{a\in A}{\downarrow} a$ and ${\uparrow}A=\bigcup_{a\in A}{\uparrow}a$.
A subset $A\subset E$ will be called {\em upper} if $A={\uparrow}A$.

Given two points $x,y$ of a topological semilattice $E$, we write $x\ll y$ if $y\in\Int({\uparrow}x)$.
For a point $x\in E$ consider the sets
$${\Uparrow}x=\{y\in E:x\ll y\}\mbox{ \ and \ }{\Downarrow}x=\{y\in E:y\ll x\},$$
and observe that ${\Uparrow}x=\Int({\uparrow}x)$.

A point $x$ of a topological semilattice $E$ will be called {\em locally minimal} if its upper cone ${\uparrow}x$ is open in $S$. In this case ${\Uparrow}x={\uparrow}x$. Observe that a point $x\in E$ is locally minimal if and only if it is isolated in its lower cone ${\downarrow}x$.

A topological semilattice $E$ is called
\begin{itemize}
\item {\em lower locally compact} if each point $x\in E$ has a closed neighborhood $\bar O_x\subset E$ such that $\bar O_x\cap{\downarrow}x$ is compact;
\item {\em Lawson} if open subsemilattices form a base of the topology of $E$;
\item a {\em $U$-semilattice} if for every open set $U$ in $E$ and every point $x\in U$ there is a point $y\in U$ such that $y\ll x$;
\item a {\em $U_2$-semilattice} open set $U$ in $E$ and every point $x\in U$ there are a point $y\in U$ and a closed-and-open ideal $I\subset E$ such that $x\in E\setminus I\subset {\uparrow}y$;
\item a {\em $U_0$-semilattice} if for every open set $U$ in $E$ and every point $x\in U$ there is a locally minimal point $y\in U$ such that $y\le x$.
\end{itemize}

These properties of topological semilattices relate as follows:
$$
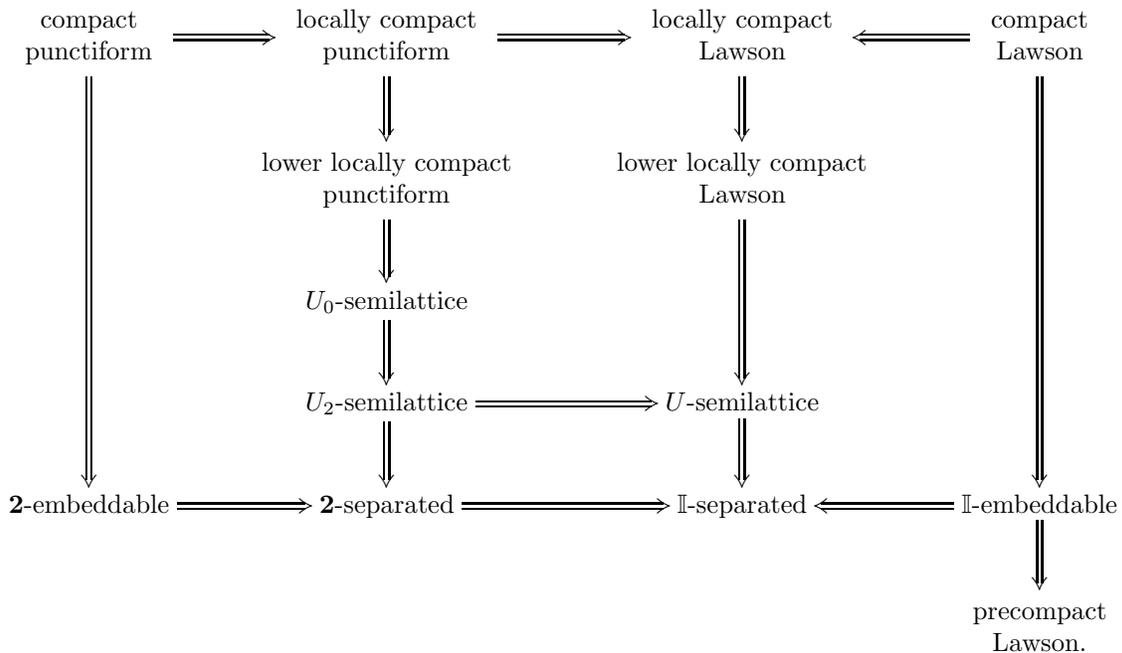
\xymatrix{
{\begin{array}{c}\mbox{compact}\\ \mbox{punctiform}\end{array}}\ar@{=>}[dddd]\ar@{=>}[r]
&{\begin{array}{c}\mbox{locally compact}\\ \mbox{punctiform}\end{array}} \ar@{=>}[d]\ar@{=>}[r]
&{\begin{array}{c}\mbox{locally compact}\\ \mbox{Lawson}\end{array}}\ar@{=>}[d]
&{\begin{array}{c}\mbox{compact}\\ \mbox{Lawson}\end{array}}\ar@{=>}[dddd]\ar@{=>}[l]\\
 &{\begin{array}{c}\mbox{lower locally compact}\\ \mbox{punctiform}\end{array}} \ar@{=>}[d]
 &{\begin{array}{c}\mbox{lower locally compact}\\ \mbox{Lawson}\end{array}}\ar@{=>}[dd]\\
&\mbox{$U_0$-semilattice}\ar@{=>}[d]\\
&\mbox{$U_2$-semilattice}\ar@{=>}[d]\ar@{=>}[r]&\mbox{$U$-semilattice}\ar@{=>}[d]\\
\mbox{$\mathbf 2$-embeddable}\ar@{=>}[r]&\mbox{$\mathbf 2$-separated}\ar@{=>}[r]&\mbox{$\II$-separated}&\mbox{$\II$-embeddable}\ar@{=>}[d]\ar@{=>}[l]\\
&&&
{\begin{array}{c}\mbox{precompact}\\ \mbox{Lawson.}\end{array}}
}
$$

Non-trivial implications from this diagram are proved in the following proposition. The last statement of this proposition says that our definition of a $U$-semilattice is equivalent to that given in \cite[p.16]{CHK2}.

\begin{proposition}\label{semlat} Let $E$ be a topological semilattice.
\begin{enumerate}
\item If $E$ is locally compact and punctiform, then $E$ is Lawson.
\item If $E$ is lower locally compact and punctiform, then $E$ is a $U_0$-semilattice;
\item If $E$ is lower locally compact and Lawson, then $E$ is a $U$-semilattice;
\item If $E$ is a $U$-semilattice, then $E$ is $\II$-separated;
\item If $E$ is a $U_2$-semilattice, then $E$ is $\mathbf 2$-separated;
\item If $E$ is a $\mathbf 2$-embeddable $U$-semilattice, then $E$ is a $U_2$-semilattice;
\item $E$ is a $U$-semilattice if and only if for every upper open set $U$ in $E$ and every point $x\in U$ there is a point $y\in U$ such that $y\ll x$.
\end{enumerate}
\end{proposition}

\begin{proof} 1. If $E$ is locally compact and punctiform, then $E$ is zero-dimensional according to \cite[1.4.5]{Eng}. To show that $E$ is Lawson, we need for every point $e\in E$ and open neighborhood $O_e\subset E$ of $e$ to find an open subsemilattice $L_e\subset O_e$ that contains $e$. By the local
compactness and zero-dimensionality of $E$, the point $e$ has a compact open neighborhood $K_e\subset O_e$. It follows that the set $U_e=\{x\in K_e:ex\in K_e\}$ is an open compact neighborhood of $e$. Observe that for every $u\in U_e$ we get $e(eu)=eu\in K_e$, which implies $eU_e\subset U_e$. Then $L_e=\{x\in K_e:xU_e\subset U_e\}\subset K_e\subset O_e$ is a required compact open subsemilattice of $S$ that contains $e$.
\smallskip

2. Assume that $E$ is lower locally compact and punctiform.
To prove that $E$ is a $U_0$-semilattice, fix any point $e\in E$ and a neighborhood $O_e\subset E$ of $e$. Since $E$ is lower locally compact, we can assume that $\bar O_e\cap{\downarrow}e$ is compact. Being punctiform, the compact space $\bar O_e\cap{\downarrow}e$ is zero-dimensional and hence contains a closed-and-open neighborhood $K_e\subset O_e\cap{\downarrow}e$ of $e$.
It follows that set $U_e=\{x\in K_e:ex\in K_e\}$ is an open compact neighborhood of $e$ in ${\downarrow}e$ and $L_e=\{x\in {\downarrow}e:xU_e\subset U_e\}\subset K_e\subset O_e$ is a compact open subsemilattice in ${\downarrow}e$. By the compactness, the semilattice $L_e$ contains the smallest element $s\in L_e\subset O_e$. Consider the retraction $r:E\to{\downarrow}e$, $r(x)=ex$, and observe that $r^{-1}(L_e)$ is an open neighborhood of $s$ in $E$, contained in  the upper cone ${\uparrow}s$ and witnessing that $e\in{\uparrow}s={\Uparrow}s$. So, $E$ is a $U_0$-semilattice.
\smallskip

3. Assume that $E$ is lower locally compact and Lawson. To prove
that $E$ is a $U$-semilattice, fix any point $e\in E$ and a
neighborhood $O_e\subset E$ of $e$. Since $E$ is lower locally
compact, we can assume that $\bar O_e\cap {\downarrow}e$ is
compact. By the regularity of the compact Hausdorff space $\bar
O_e\cap{\downarrow}e$, the point $e$ has a compact neighborhood
$K_e\subset O_e\cap{\downarrow}e$. Since the topological
semilattice $E$ is Lawson, the point $e$ is contained in an open
subsemilattice $L_e\subset O_e$ such that
$L_e\cap{\downarrow}e\subset K_e$. Then the subsemilattice
$L_e\cap {\downarrow}e$ has compact closure in $E$ and hence
contains the smallest element $s\in K_e\subset O_e$. By analogy
with the preceding proof we can show that $e\in{\Uparrow}s$, which
means that $E$ is a $U$-semilattice.
\smallskip

4. The forth statement was proved in Theorem 2.11 of \cite{CHK2}.
\smallskip

5. Assume that $E$ is a $U_2$-semilattice. Take any two distinct points $x,y\in E$ and consider their product $xy\in E$.
We lose no generality assuming that $y\ne xy$. Since $U=\{u\in E:u\ne xu\}$ is an open set containing the point $y$,
by the definition of a $U_2$-semilattice, there is a point $u\in U$ and a closed-and-open ideal $I\subset E$ such
that $y\subset E\setminus I\subset {\uparrow}u$.
It follows that the function $h:E\to\mathbf 2$ defined by
$$h(e)=\begin{cases}
0&\mbox{if $e\in I$},\\
1&\mbox{if $e\notin I$}
\end{cases}
$$is a continuous homomorphism such that $h(x)=0$ and $h(y)=1$. Therefore, $E$ is $\mathbf 2$-separating.
\smallskip

6. Assume that $E$ is a $\mathbf 2$-embeddable $U$-semilattice. To show that $E$ is a $U_2$-semilattice, fix any point $x\in E$ and an open neighborhood $O_x\subset E$ of $x$. Since $E$ is a $U$-semilattice, there is an idempotent $e\in O_x$ such that $x\in{\Uparrow}e$. Since $E$ is $\mathbf 2$-embeddable, the topology of $E$ is generated by the subbase consisting of the sets $h^{-1}(t)$ where $t\in\mathbf 2$ and $h:E\to \mathbf 2$ is a continuous homomorphism. Consequently, we can find continuous homomorphisms $h_1,\dots,h_n:E\to\mathbf 2$ and points $t_1,\dots,t_n\in\mathbf 2$ such that $x\in \bigcap_{i=1}^n h_i^{-1}(t_i)\subset{\Uparrow}e$. Taking into account that
${\Uparrow}e=\Int({\uparrow}e)$ is an upper set in $E$, we conclude that $$x\in \bigcap_{i=1}^n h_i^{-1}(t_i)\subset\bigcap_{i=1}^nh_i^{-1}(\{t_i,1\})\subset {\Uparrow}e$$
and hence $I=\bigcup_{i=1}^n h_i^{-1}(\mathbf 2\setminus\{t_i,1\})$ is a required open-and-closed ideal in $E$ such that $x\in E\setminus I\subset{\Uparrow}e$.
\smallskip

7. The ``only if'' part of statement (7) is trivial. To prove the ``if'' part, assume that 
for every upper open set $U$ in $E$ and every point $x\in U$ there is a point $y\in U$ such that $y\ll x$. Given any open set $V\subset E$ and a point $x\in V$, consider the upper set ${\upa}V=\bigcup_{v\in V}{\upa}v$ and observe that it is open in $E$ as ${\upa}V=\{x\in E:\exists v\in V\;xv\in V\}$. By our assumption, there is a point $y\in{\upa}V$ such that $y\ll x$. For the point $y$ there is a point $v\in V$ such that $v\le y$. Taking into account that $x\in \Int({\upa}y)\subset\Int({\upa}v)$, we conclude that $v\ll x$. So, $v\in V$ is a required point witnessing that $E$ is a $U$-semilattice.
\end{proof}

A subset $A$ of a topological semilattice $E$ will be called {\em $U$-dense in $E$} if for each point $x\in E$ and a neighborhood $O_x\subset E$ of $x$ there is an idempotent $e\in O_x\cap A$ such that $e\ll x$. It is clear that a topological semilattice $E$ is a $U$-semilattice if and only if the set $E$ is $U$-dense in $E$.

A topological semilattice $E$ will be called {\em $U$-separable} if it contains a countable $U$-dense subset $A\subset E$.

\begin{proposition}\label{pVsep} Each second countable $U$-semilattice $E$ is $U$-separable.
\end{proposition}

\begin{proof} Fix a countable base $\mathcal B$ of the topology of the space $E$. For every basic set $B\in\mathcal B$ we shall construct a countable set $A_B\subset B$ such that for any point $x\in B$
there is a point $a\in A_B$ with $x\in{\Uparrow}a$.

Since $E$ is a $U$-semilattice, for every  point $x\in B$, there are a point $a_{x}\in B$ with $x\in{\Uparrow}a_{x}$ and a basic neighborhood $U_{x}\in\mathcal B$ of $x$ such that $U_{x}\subset{\Uparrow}a_{x}$. The family $\U_B=\{U_{x}:x\in B\}\subset\mathcal B$ is countable and hence can be enumerated as $\U_B=\{V_n:n\in\w\}$. For every $n\in\w$ find a point $x_n\in B$ such that $V_n=U_{x_n}$ and observe that $A_B=\{a_{x_n}:n\in\w\}$ is a countable set with the required property: for any point $x\in B$
there is a point $a=a_{x_n}\in A_B$ such that $x\in U_x=V_n=U_{x_n}\subset {\Uparrow}a$.

Then the countable union $A=\bigcup_{B\in\mathcal B}A_B$ is a countable $U$-dense subset in $E$, which implies that the semilattice $E$ is $U$-separable.
\end{proof}

\subsection{Topological Clifford $U$-semigroups}
A topological Clifford semigroup $S$ is defined to be a {\em
topological Clifford $U$-semigroup} if its maximal semilattice
$E=\{x\in S:xx=x\}$ is a $U$-semilattice.

\subsection{Reduced products over topological semigroups}
Let $E$ and $H$ be topological semigroups and $I$ be a
closed ideal in $E$.

By the {\em reduced product} $E\times_{I}H$ of $E$ and $H$ over the ideal $I$ we mean
the set $I\cup ((E\setminus I)\times H)$ endowed with the smallest topology
such that
\begin{itemize}
\item the map $(E\setminus I)\times H\hookrightarrow E\times_{I}H$
is a topological embedding,
\item the projection
$\pi:E\times_{I}H\rightarrow E$ is continuous.
\end{itemize}
The semigroup operation on $E\times_I H$ can be defined as a unique binary operation on $E\times_I H$
 such that the projection $q:E\times H\to E\times_I H$ defined by
$$q(x,y)=\begin{cases}
x&\mbox{if $x\in I$};\\
(x,y)&\mbox{otherwise}
\end{cases}
$$
is a semigroup homomorphism. A routine verification shows that the reduced product $E \times_{I}H$ of topological semigroups is a topological semigroup.
Moreover, if the semigroups $E,H$ are inverse (Clifford), then
so is the semigroup $E\times_I H$. The homomorphisms
$$\xymatrix{E\times H\ar^{q}[r]&E\times_I H\ar^-{\pi}[r]&E}$$
will be called {\em natural projections}.

We shall be especially interested in reduced products of the form $E\times_{I_e}H$ where $H$ is a topological group, $E$ is a topological semilattice, $e$ is an idempotent in $E$, and $I_e=E\setminus {\Uparrow}e$ is the closed ideal in $E$ determined by the idempotent $e$. Here ${\Uparrow}e=\Int({\uparrow}e)$ is the interior of the upper cone of $e$ in $E$.

For the compact topological semilattice $\II=[0,1]$ and the idempotent $e=0$ the ideal $I_e$ coincides with the singleton $\{0\}$. The reduced product $\II\times_{\{0\}} G$ will be denoted by $\widehat G$ and called the {\em cone} over the topological group $G$. The cone $\widehat{G}=\II\times_{\{0\}} G$ contains the reduced product $\mathbf 2\times_{\{0\}} G$, which will be denoted by $\dot G$ and called the {\em 0-extension} of $G$. Observe that $\hat G$ and $\dot G$ are topological Clifford $U$-semigroups.

\subsection{Ditopological Clifford semigroups}\label{ditop}
 For two subsets $A,B$ of a semigroup $S$ let
$$B\div A=\{x\in S:\exists b\in B\;\exists a\in A\;\;b=xa\}$$
be the results of right division of $B$ by $A$.

Let $S$ be a topological Clifford semigroup and $\pi:S\to E$, $\pi:x\mapsto xx^{-1}=x^{-1}x$, be the projection of $S$ onto its maximal semilattice $E=\{x\in S:xx=x\}$.
Following \cite{BP}, we define a topological Clifford semigroup $S$ to be a {\em ditopological Clifford semigroup} if for any point $x\in S$ and a neighborhood $O_x\subset S$ of $x$ there are neighborhoods $U_x\subset S$ and $W_{\pi(x)}\subset E$ of $x$ and $\pi(x)$, respectively, such that $(U_x\div  W_{\pi(x)})\cap\pi^{-1}(W_{\pi(x)})\subset O_x$.
Ditopological Clifford semigroups are particular cases of ditopological unosemigroups introduced and studied in \cite{BP}, where the following proposition is proved.

\begin{proposition}\label{p2.5} The class of ditopological Clifford semigroups contains all compact topological Clifford semigroups, all topological groups,
all topological semilattices, and is closed under taking Clifford subsemigroups, Tychonoff products, and reduced products.
\end{proposition}

This proposition implies:

\begin{corollary}\label{c2.6} Let $E$ be a topological semigroup, $I$ be a closed ideal in $E$, and $G$ be a topological group. Then the reduced product $E\times_I G$ is a ditopological Clifford semigroup. In particular, for every idempotent $e\in E$ the reduced product $E\times_{I_e} G$ is a ditopological Clifford semigroup. Consequently, the $0$-extension $\dot G=\mathbf 2\times_{\{0\}} G$ and the cone $\widehat G=\II\times_{\{0\}} G$ over the topological group $G$ are ditopological Clifford semigroups.
\end{corollary}

\section{Two Embedding Theorems}\label{2emb}

In this section we prove two embedding theorems for ditopological Clifford $U$-semigroups. For compact topological Clifford semigroups these embedding theorems were proved by O.~Hryniv \cite{Hryn}. We recall that a {\em topological Clifford $U$-semigroup} is a topological Clifford semigroup $S$ whose maximal semilattice $E$ is a $U$-semilattice. The map $\pi:S\to E$, $\pi:x\mapsto xx^{-1}=x^{-1}x$, will be called the projection of $S$ onto its maximal semilattice $E$.
For every idempotent $e\in E$ the preimage $\pi^{-1}(e)$ coincides with the maximal subgroup $H_e$ of $S$ containing the idempotent $e$.

\subsection{The First Embedding Theorem} Given any idempotent $e\in E$, consider the ideal $I_e=E\setminus{\Uparrow}e=E\setminus\Int({\uparrow}e)$ in $E$ and the reduced product $E\times_{I_e} H_e$. Let $\pi_e:E\times_{I_e} H_e\to E$ denote the natural projection.

Consider the continuous semigroup homomorphism $h_{e}:S
\rightarrow E\times_{I_e}H_{e}$ defined by the formula
$$
h_{e}(x)=
\begin{cases}
\big(\pi(x), xe\big),&\mbox{if }\pi(x)\in{\Uparrow}e,\\
\pi(x),& \mbox{otherwise}.
\end{cases}
$$
For any subset $A\subset E$ the homomorphisms $h_e$, $e\in A$, compose a continuous homomorphism
$$h_A=(h_e)_{e\in A}:S\to \prod_{e\in A}E\times_{I_e}H_e,\;\;h_A:x\mapsto(h_e(x))_{e\in A}.$$

\begin{theorem}\label{embedding} If $S$ is a ditopological Clifford $U$-semigroup,
then for every $U$-dense subset $A\subset E$ the homomorphism $$h_A:S\rightarrow
\prod_{e\in A}E\times_{I_{e}}H_{e}$$ is a topological embedding.
\end{theorem}

\begin{proof} We lose no generality assuming that $S$ is not empty.
In this case the semilattice $E$ and the $U$-dense subset $A$ of $E$
both are non-empty. The map $h_A$ is a continuous homomorphism,
being a diagonal product of continuous homomorphisms $h_{a}, a\in
A$.

First, we show that $h_A$ is an injective homomorphism. Take two
distinct points $x,y\in S$ and consider their projections $\pi(x),
\pi(y)$ on the maximal semilattice $E$.

If $\pi(x)\neq \pi(y)$, then it follows immediately that
$h_{a}(x)\neq h_{a}(y)$ for all $a\in A$, which implies
$h_A(x)\neq h_A(y)$.

If $\pi(x)= \pi(y)$, then the open set $V=\{e\in E: xe\neq ye\}\subset E$ is a neighborhood of the idempotent $e=\pi(x)=\pi(y)$.
The $U$-density of the set $A$ in $E$ yields an
element $a\in A\cap V$ such that $e\in {\Uparrow} a$. For this
element $a$ we get $h_{a}(x)=(e,xa)\neq(e,ya)=h_{a}(y)$, which
means that $h_{a}$ separates the points $x,y\in S$. Hence $h_A$ is
injective.

Now we prove that the inverse function $h_A^{-1}:h_A(S)\rightarrow
S$ is continuous. Given any element $x\in S$ and an open
neighborhood $O_{x}\subset S$ of $x$, we need to find a
neighborhood $O_{h_A(x)}\subset h_A(S)$ of the point $h_A(x)$ such that $h_A^{-1}(O_{h_A(x)})\subset O_{x}$.

Since the topological Clifford semigroup $S$ is ditopological, there are neighborhoods $U_{x}$ of $x$
and $W_{e}\subset E$ of the idempotent $e=\pi(x)$ such that
$(U_x\div  W_e)\cap \pi^{-1}(W_{e})\subset
O_x$. It follows from $xe=x\in U_{x}$ and the continuity of the
multiplication that there is a neighborhood $W'_{e}\subset W_{e}$
of the idempotent $e$ such that $x W'_{e}\subset
U_{x}$.

The $U$-density of $A$ in $E$ yields a point $a\in A\cap W_e'$ such that $e\in {\Uparrow}a$. Consider the neighborhood $W=W'_{e}\cap {\Uparrow}a$
of $e$ and the open subset $U=U_{x}\cap H_{a}$ in the maximal group $H_{a}$. Observe that $W\times U$ is an open neighborhood of the
point $h_{a}(x)$ in $E\times_{I_{a}}H_{a}$ and hence
$O_{h_A(x)}=\pr^{-1}_{a}(W\times U)$ is an open neighborhood of
$h_A(x)$ in the Tychonoff product $\prod_{a\in
A}E\times_{I_{a}}H_{a}$. Here $\pr_{a}:\prod_{b\in
A}E\times_{I_{b}}H_{b}\rightarrow E\times_{I_{a}}H_{a}$ denotes the $a$-th coordinate projection.

To finish the proof of the continuity of $h_A^{-1}$ at $h_A(x)$ it suffices to check that $h_A^{-1}(O_{h_A(x)})\subset O_{x}$.
Take any point $y\in S$ with $h_A(y)\in O_{h_A(x)}$.
The definition of $O_{h_A(x)}$ guarantees $h_{a}(y)=\pr_a(h_A(y))\in W\times U$
and hence $\pi(y)\in W\subset W'_{e}\subset W_{e}$ and $ya\in
U\subset U_{x}$, which means that $y\in (U_x\div
W_e)\cap \pi^{-1}(W_{e})\subset  O_x$.
\end{proof}

\subsection{The Second Embedding Theorem} Now we shall modify the embedding $h_A$ from Theorem~\ref{embedding} to produce an embedding
of the topological Clifford $U$-semigroup $S$ into the Tychonoff product of the maximal semilattice $E$ and the cones $\hat H_e$, $e\in E$, over  maximal subgroups of $S$.

Denote by $\Hom(E,\II)$ the set of all continuous homomorphisms from the maximal semilattice $E$ of $S$ to the min-interval $\II$.
The set $\Hom(E,\II)$ contains a subset $\Hom(E,\mathbf 2)$ consisting of all homomorphisms from $E$ to the two-element semilattice $\mathbf 2=\{0,1\}$. For two idempotents $e,a\in E$ put
$$
\begin{aligned}
\Hom_e^a(E,\II)&=\big\{h\in \Hom(E,\II):h(a)=1\mbox{ \ and \ }h(I_e)\subset\{0\}\big\}\mbox{ \ and}\\
\Hom_e^a(E,\mathbf 2)&=\Hom_e^a(E,\II)\cap\Hom(E,\mathbf 2).
\end{aligned}
$$
By Lemma 2.10 in \cite{CHK2}, for any idempotents $e\ll a$ of $E$ the set  $\Hom_e^a(E,\II)$ is not empty. If $E$ is a $U_2$-semilattice, then moreover, the set $\Hom_e^a(E,\mathbf 2)$ is not empty.

For any idempotents $e\ll a$, fix a continuous homomorphism $h_e^a\in\Hom_e^a(E,\II)$ such that
$h_e^a\in\Hom_e^a(E,\mathbf 2)$ if $\Hom_e^a(E,\mathbf 2)$ is not empty. The latter condition holds if $E$ is a $U_2$-semilattice. The homomorphism $h_e^a:E\to\II$ induces a continuous homomorphism $\hat h_e^a:S\to\widehat H_e$ defined by the formula
$$\hat h_e^a(x)=\begin{cases}
(h_e^a\circ \pi(x),xe),&\mbox{if $h_e^a\circ\pi(x)>0$},\\
0,&\mbox{otherwise}.
\end{cases}
$$
If $h_e^a\in\Hom(E,\mathbf 2)$, then $\hat h_e^a(S)\subset \dot
H_e\subset \widehat H_e$.

For any subset $A\subset E$, the homomorphisms $\hat h_e^a$, $e\in A$, $a\in A\cap{\Uparrow}e$, compose the continuous homomorphism
$$\hat h_A^A:S\to \prod_{e\in A}\widehat H_e^{A\cap{\Uparrow}e},\;\;\;\hat h_A^A:x\mapsto \big((\hat h_e^a(x))_{a\in A\cap{\Uparrow}e}\big)_{e\in A}.$$
Taking the diagonal product of $\hat h_A^A$ with the natural projection $\pi:S\to E$, we obtain a continuous homomorphism
$$\pi\hat h_A^A:S\to E\times \prod_{e\in A}\widehat H_e^{A\cap{\Uparrow}e},\;\;\;\pi\hat h_A^A:x\mapsto \big(\pi(x),\hat h_A^A(x)\big).$$

\begin{theorem}\label{T2} If $S$ is a ditopological Clifford $U$-semigroup, then for any $U$-dense subset $A$ in $E$ the homomorphism
$$\pi\hat h_A^A:S\to E\times\prod_{e\in A}\widehat H_e^{A\cap{\Uparrow}e}$$is a topological embedding. If $E$ is a $U_2$-semilattice,
then $$\pi\hat h_A^A(S)\subset E\times \prod_{e\in A}\dot H_e^{A\cap{\Uparrow}e}.$$
\end{theorem}

\begin{proof}
The definition of the map $\pi\hat h_A^A$ guarantees that it is a continuous homomorphism. To prove that it is injective
take any two distinct points $x,y\in S$ and consider their projections
$\pi(x),\pi(y)$ on the maximal semilattice $E$.

If $\pi(x)\neq \pi(y)$, then $\pi\hat{h}_A^{A}(x)=(\pi(x),\hat{h}_A^{A}(x))\neq (\pi(y),\hat{h}_A^{A}(y))=\pi\hat{h}_A^{A}(y)$.

If $\pi(x)=\pi(y)$, then the open set $V=\{e\in E: xe\neq ye\}\subset E$ is a neighborhood of the idempotent $e=\pi(x)=\pi(y)$. The $U$-density of the set $A$ in $E$ yields an
element $a\in A\cap V$ with $e\in {\Uparrow}a$, and an element $b\in A\cap V\cap{\Uparrow}a$ with $e\in{\Uparrow}b$.
It follows that $1=h^{b}_{a}(b)=h_a^b(e)$, which clearly forces
$\hat{h}_{a}^{b}(x)=({h}^{b}_{a}(e),ax)=(1,ax)\neq (1,ay)=({h}^{b}_{a}(e),ay)=\hat{h}_{a}^{b}(y)$ and so $\pi\hat{h}_A^{A}(x)\neq \pi\hat{h}_A^{A}(y)$.

It remains to show that the inverse function
$(\pi\hat{h}_A^{A})^{-1}:\pi\hat{h}_A^{A}(S)\rightarrow S$ is
continuous. Given an element $x\in S$ and an open neighborhood
$O_x\subset S$ of $x$, we need to find a neighborhood
$O_{y}\subset E\times \prod_{e\in A}\widehat
H_e^{A\cap{\Uparrow}e}$ of the point $y=\pi\hat h_A^A(x)$ such
that $(\pi\hat{h}_A^{A})^{-1}(O_y)\subset O_x$.

Since the topological Clifford semigroup $S$ is dicontinuous, there are a neighborhood $U_x\subset S$ of $x$ and a neighborhood
$W_{e}\subset E$ of the idempotent $e=\pi(x)$ such that $(U_x\div  W_e)\cap \pi^{-1}(W_{e})\subset O_x$.
It follows from $xe=x\in U_x$ and the continuity of multiplication that there is a neighborhood $W'_{e}\subset W_{e}$ such that $x W_{e}'\subset U_x$.

The $U$-density of $A$ in $E$ yields a point $a\in A\cap W_e'$ such that $e\in {\Uparrow}a$ and a point $b\in A\cap W_e'\cap{\Uparrow}a$ such that $e\in{\Uparrow}b$.
It follows that $W=W_e'\cap{\Uparrow}b$ is an open neighborhood of the point $e=\pi(x)$.

Observe that $h_a^b(e)=h_a^b(b)=1$ and hence $\hat
h_a^b(x)=(1,xa)$. Then $U'=(0,1]\times(U_x\cap H_a)$ is an open
neighborhood of $\hat h_a^b(x)$ in $\widehat H_a$ and its preimage
$U=\pr_a^{-1}(\pr_b^{-1}(U'))$ is an open neighborhood of $\hat
h_A^A(x)$ in the Tychonoff product $\prod_{c\in
A}\widehat{H}^{A\cap {\Uparrow}c}_c$. Here $pr_{a}:\prod_{c\in
A}\widehat{H}^{A\cap {\Uparrow}c}_c\rightarrow \widehat{H}^{A\cap
{\Uparrow}{a}}_{a}$ and $pr_{b}:\widehat{H}^{A\cap
{\Uparrow}{a}}_{a}\rightarrow \widehat{H}_{a}$ denote the
coordinate projections.

We claim that the neighborhood $O_{y}=W\times U$ of the point $y=\pi\hat{h}_A^{A}(x)$ witnesses that $(\pi\hat{h}_A^{A})^{-1}$ is continuous at $y$.
Given any point $z\in S$ with $\pi\hat h_A^A(z)\in O_y=W\times U$, we need to check that $z\in O_x$.
It follows that $\pi(z)\in W=W'_e\cap{\Uparrow}b\subset W_e$ and $(h_a^b(\pi(z)),za)=\hat h_a^b(z)\in U'=(0,1]\times (U_x\cap H_a)$.
So, $z\in (U_x\div  W_e)\cap\pi^{-1}(W_e)\subset O_x$,  which completes the proof.
\end{proof}

\section{Some corollaries of Embedding Theorem}

In this section we shall derive some corollaries of the Embedding
Theorem \ref{T2}.

\begin{corollary}\label{semgroup} Let $S$ be a topological Clifford $U$-semigroup and $A$ be a $U$-dense subset in $E$.
Then the following conditions are
equivalent:
\begin{enumerate}
\item{$S$ embeds into the Tychonoff product of topological
semilattices and cones over topological groups};
\item{$S$ embeds into the topological Clifford semigroup $E\times
\prod_{e\in A}\widehat{H}^{{\Uparrow}e\cap A}_{e}$};
\item{$S$ is ditopological}.
\end{enumerate}
\end{corollary}

\begin{proof} The implication $(3)\Ra(2)$ follows from Theorem~\ref{T2}, $(2)\Ra(1)$ is trivial,
and $(1)\Ra(3)$ follows from Proposition~\ref{p2.5} and Corollary~\ref{c2.6}.
\end{proof}

In the same way Theorem~\ref{T2} implies:

\begin{corollary}\label{semgroup1} Let $S$ be a topological Clifford semigroup whose idempotent band $E$ is a $U_2$-semilattice,
and $A$ be a $U$-dense subset in $E$. Then the following conditions are
equivalent:
\begin{enumerate}
\item{$S$ embeds into the Tychonoff product of topological
semilattices and $0$-extensions of topological groups};
\item{$S$ embeds into the topological Clifford semigroup $E\times
\prod_{e\in A}\dot{H}^{{\Uparrow}e\cap A}_{e}$};
\item{$S$ is ditopological}.
\end{enumerate}
\end{corollary}

Let us recall that given two topological semigroups $X,Y$, we say that $X$ is {\em $Y$-embeddable}
if the canonical homomorphism $X\to Y^{\Hom(X,Y)}$ is a topological embedding. It is easy to see that $X$ is $Y$-embeddable,
if and only if $X$ embeds into some power $Y^\kappa$ of $Y$. In this case we shall write $X\hookrightarrow Y^\kappa$.

\begin{corollary}\label{cones} Let $S$ be a topological Clifford $U$-semigroup, $E$ be its maximal semilattice, and
  $H=\prod_{e\in E}H_e$ be the Tychonoff product of its maximal subgroups.
The following conditions are equivalent:
\begin{enumerate}
\item $S$ embeds into a Tychonoff product of the cones over
topological groups; \item $S$ is $\widehat H$-embeddable; \item
$S$ is ditopological and $E$ is $\II$-embeddable.
\end{enumerate}
\end{corollary}

\begin{proof} The implication $(2)\Ra(1)$ is trivial and $(1)\Ra(3)$ follows from Proposition~\ref{p2.5} and Corollary~\ref{c2.6}.

$(3)\Ra(1)$ Assume that the topological Clifford $U$-semigroup $S$
is ditopological and its maximal semilattice $E$ is
$\II$-embeddable. Consequently,
$E\hookrightarrow\II^{\Hom(E,\II)}\hookrightarrow \widehat
H^{\Hom(E,\II)}$.

 By Theorem~\ref{T2}, $S$ embeds into the Tychonoff product $E\times \prod_{e\in E}\widehat H_e^{E}$,
 and the latter product embeds into the Tychonoff product $\widehat H^{\Hom(E,\II)}\times \widehat H^{E\times E}$, which implies that $S$ is
  $\widehat H$-embeddable.
\end{proof}

In the same way Theorem~\ref{T2} implies:

\begin{corollary} Let $S$ be a topological Clifford $U$-semigroup, $E$ be its maximal semilattice, and
  $H=\prod_{e\in E}H_e$ be the Tychonoff product of its maximal subgroups.
The following conditions are equivalent:
\begin{enumerate}
\item $S$ embeds into a Tychonoff product of the 0-extensions of topological groups;
\item $S$ is $\dot H$-embeddable;
\item $S$ is ditopological and $E$ is $\mathbf 2$-embeddable.
\end{enumerate}
\end{corollary}

\section{Characterizing precompact topological Clifford $U$-semigroups}\label{precomp}

In this section we shall apply Embedding Theorem~\ref{T2} to prove a compactification theorem for topological Clifford $U$-semigroups.
We recall that a topological semigroup $S$ is called {\em precompact} if $S$ embeds into a compact topological  semigroup.

It is well-known \cite[3.7.16]{AT} that a topological group $G$ is precompact if and only if $G$ is {\em totally bounded},
which means that for each non-empty open set $U\subset G$ there is a finite subset $F\subset G$ such that $FU=G=UF$.

Theorem~\ref{T2} combined with Proposition~\ref{p2.5} imply the following characterization of precompact topological Clifford $U$-semigroups.

\begin{theorem}\label{precomp} A topological Clifford $U$-semigroup $S$ is precompact if and only if the following conditions are satisfied:
\begin{enumerate}
\item the maximal semilattice $E=\{e\in S:ee=e\}$ of $S$ is precompact,
\item every maximal subgroup $H_e$, $e\in E$, of $S$ is precompact, and
\item $S$ is ditopological.
\end{enumerate}
\end{theorem}

\section{Metrizability of topological Clifford $U$-semigroups}

In this section we apply the Embedding Theorem~\ref{T2} to construct subinvariant metrics on topological Clifford $U$-semigroups.

A metric $d$ on a Clifford semigroup $S$ will be called
\begin{itemize}
\item {\em left subinvariant} if $d(zx,zy)\le d(x,y)$ for any points $x,y,z\in S$;
\item {\em right subinvariant} if $d(xz,yz)\le d(x,y)$ for any points $x,y,z\in S$;
\item {\em subinvariant} if $\max\{d(zx,zy),d(xz,yz)\}\le d(x,y)=d(x^{-1},y^{-1})$ for any points $x,y,z\in S$.
\end{itemize}
 We shall say that a topological Clifford semigroup $S$ is {\em metrizable} ({\em by a subinvariant metric}) if the topology of $S$ is generated by some (subinvariant) metric.

By the Birkhoff-Kakutani Theorem~\cite[3.3.12]{AT}, a topological group $G$ is metrizable by a left subinvariant metric if and only if $G$ is first countable. By \cite[3.3.14]{AT}, a topological group $G$ is metrizable by a subinvariant metric if and only if $G$ is first countable and {\em balanced}. The latter means that for every neighborhood $U\subset G$ of the unique idempotent $e$ of $G$ there is a neighborhood $V\subset U$ of $e$ such that $xVx^{-1}=V$ for all $x\in G$. A simple example of a metrizable topological group, which is not balanced (and hence not metrizable by a subinvariant metric) is the group
$\Aff(\IR)$ of affine transformations of the real line.

\begin{theorem}\label{t6.1} Each second countable precompact topological Clifford semigroup $S$ is metrizable by a subinvariant metric.
\end{theorem}

\begin{proof} By Proposition~2.3, $S$ embeds into a second countable compact topological Clifford semigroup $K$, which is metrizable by some metric $d$. The continuity of the semigroup operation and the inversion on the compact space $K$ implies that the metric
$$\rho(x,y)=\max_{z\in K}\max\big\{d(zx,zy),d(xz,yz),d(zx^{-1},zy^{-1}), d(x^{-1}z,y^{-1}z)\big\}$$
is a well-defined subinvariant continuous metric on $K$.
By the compactness of $K$, this metric generates the topology of $K$ and its restriction $\rho|S\times S$ is a subinvariant metric generating the topology of $S$.
\end{proof}

Now we turn to the metrization problem for ditopological Clifford $U$-semigroups.
With help of Theorem~\ref{T2} this problem can be reduced to the problem of metrization of semilattices and cones over topological groups.

The latter problem is quite simple. Assume that a topological
group $H$ is metrizable by a metric $d$. Then the cone $\widehat
H$ over $H$ is metrizable by the metric
$$\hat d(x,y)=\min\{ t_{x}+t_{y},|t_{x}-t_{y}|+d(h_{x},h_{y})\},\;\;x,y\in \hat H,$$ where $(t_x,x),(t_y,y)\in[0,1]\times H$ are any pairs such that $x=q(t_x,x)$ and $y=q(t_y,y)$. Here $q:[0,1]\times H\to\widehat H$ is the canonical projection. If the metric $d$ on $H$ is
(left or right) subinvariant, then so is the metric $\hat d$ on
$\widehat H$.

This fact combined with Theorem~\ref{T2} implies the following metrization theorem, which generalizes some metrization theorems proved in \cite{Ban}.

\begin{theorem}\label{metr} Assume that $S$ is a ditopological topological Clifford $U$-semigroup and $A$ is a countable $U$-dense subset on $S$.
The topological semigroup $S$ is
\begin{enumerate}
\item metrizable if and only if the maximal semilattice $E$ is metrizable and all maximal
subgroups $H_e$, $e\in A$, of $S$ are first countable;
\item metrizable by a left subinvariant metric if and only if the maximal semilattice $E$ is metrizable by a subinvariant metric and all maximal
subgroups $H_e$, $e\in A$, of $S$ are first countable;
\item metrizable by a subinvariant metric if and only if the maximal semilattice $E$ is metrizable by a subinvariant metric and all maximal
subgroups $H_e$, $e\in A$, of $S$ are first countable and balanced.
\end{enumerate}
\end{theorem}

Let us remark that by Proposition~\ref{pVsep}, each second countable $U$-semilattice contains a countable $U$-dense subset $A\subset E$.

Theorem~\ref{metr} reduces the problem of (subinvariant) metrizability of a topological Clifford $U$-semigroup $S$ to the problem of (subinvariant)
metrizablity of the maximal subsemilattice $E$ of $S$. Surprisingly, but this problem is not trivial as witnessed by the following simple example.

\begin{example}\label{e6.4} There exists a topological semilattice $E$ having the following properties:
\begin{enumerate}
\item $E$ is countable and locally compact;
\item $E$ is Lawson;
\item $E$ is metrizable;
\item $E$ is not metrizable by a subinvariant metric;
\item $E$ is not precompact.
\end{enumerate}
\end{example}

\begin{proof}
Consider a semilattice $E=\{0\}\cup \{\frac{1}{n}\}_{n\in \IN}$ endowed with the semilattice operation of minimum.
Endow $E$ with the topology $\tau$ in which all
non-zero points are isolated and the sets
$B_n=\{0\}\cup \{\frac{1}{2k}\}_{k\geqslant n}$, $n\in\IN$, form a neighborhood base at $0$. It is clear that $E$ is a topological semilattice satisfying the conditions (1)--(3) of Example~\ref{e6.4}.

Assume that $E$ is metrizable by a subinvariant metric $d$.
Then for every $n\in\IN$ and the points $x=0$, $y=\frac1{2n}$, $z=\frac1{2n+1}$, we get
$$d(0,\tfrac1{2n+1})=d(zx,zy)\le d(x,y)=d(0,\tfrac1{2n})\underset{n\to\infty}{\longrightarrow}0,$$ which implies that the sequence $\frac1{2n+1}$ tends to zero. But this contradicts the choice of the topology on $E$.
Therefore, $E$ cannot be metrizable by a subinvariant metric. By Theorem~\ref{t6.1}, the second countable topological semilattice $E$ is not precompact.
\end{proof}

\end{document}